\documentclass[11pt]{amsart}

\usepackage[dvipdfm]{hyperref}

\usepackage{mathptmx}
\usepackage{amsmath}
\usepackage{amssymb}
\usepackage{amsthm}
\usepackage{array}
\usepackage[all,tips]{xy}
\usepackage{enumerate}
\usepackage{graphicx}
\usepackage{lmodern}
\usepackage{mathrsfs}

\hypersetup{colorlinks,citecolor=blue,linkcolor=blue}

\newtheorem{theorem}{Theorem}[section]
\newtheorem{lemma}[theorem]{Lemma}
\newtheorem{prop}[theorem]{Proposition}

\theoremstyle{definition}

\theoremstyle{remark}
\newtheorem{rmk}[theorem]{Remark}

\numberwithin{equation}{section}

\DeclareMathOperator{\SL}{SL}
\DeclareMathOperator{\GL}{GL}

\DeclareMathOperator{\PU}{PU}

\DeclareMathOperator{\PSL}{PSL}
\DeclareMathOperator{\PGL}{PGL}

\DeclareMathOperator{\vol}{vol}
\newcommand{\tors}{\mathrm{tors}}

\newcommand{\Z}{\mathbb Z}
\newcommand{\Q}{\mathbb Q}
\newcommand{\R}{\mathbb R}
\newcommand{\C}{\mathbb C}

\newcommand{\K}{\mathcal K}

\renewcommand{\O}{\mathcal O}

\newcommand{\bs}{\backslash}
\newcommand{\del}{\partial}

\begin{document}
\title[Torsion in the $K$-theory of imaginary quadratic fields]{On the torsion part in the $K$-theory of imaginary quadratic fields}

\begin{abstract}
  We obtain upper bounds for the torsion in the $K$-groups of the rings of integers of
  imaginary quadratic number fields, in terms of their discriminants. 
  \end{abstract}

\author{Vincent Emery}
\thanks{The author is Supported by Swiss National Science Foundation, Project number PZ00P2-148100}

\address{
EPFL -- SB -- MATHGEOM\\
B\^atiment MA, Station 8\\
CH-1015 Lausanne\\
Switzerland
}
\email{vincent.emery@math.ch}

\date{\today}

%\subjclass{22E40 (primary); 11E57, 20G30, 51M25 (secondary)}
% Group Theory;  Geometric Topology, Number Theory

\maketitle

%\nocite{Gel04}

\section{Introduction}
\label{sec:intro}

Let $F$ be a number field with ring of integers $\O_F$, and let 
$K_n(\O_F)$ denote Quillen's $K$-theory group of degree $n$. By a theorem of
Quillen $K_n(\O_F)$ is finitely generated, and its rank was computed
by Borel for each $n \ge 2$ \cite{Bor74}.  
In this paper we are interested in the problem of finding  upper
bounds for the torsion part of $K_n(\O_F)$.
We obtain a result in the case of imaginary quadratic fields:

\begin{theorem}
  Let $n \ge 2$.
 There exists a constant $C(n)$ such that for any imaginary quadratic number field
 $F$, the group $K_n(\O_F)$ contains no $p$-torsion element for 
 any prime $p$ with
 \begin{align}
   \log(p) > C(n) |D_F|^{2n(n+1)},
  \label{eq:main-thm} 
 \end{align}
 where $D_F$ denotes the discriminant of $F$.
  \label{thm:main}
\end{theorem}

For $n = 2$ the better estimate
\begin{align}
  \log \left| K_2(\O_F) \otimes \Z\big[ 1/6 \big] \right| &\le C \cdot |D_F|^2 \log|D_F| 
  \label{eq:K2}
\end{align}
was obtained in \cite{EmeK2}. Tensoring by $\Z[1/6]$ excludes the $2$-
and $3$-torsion in the bound. In a similar way,
Theorem~\ref{thm:main} is obtained from an upper bound for $|\tors\
K_n(\O_F)|$
that holds modulo small torsion, although we need to exclude more
primes here (see Proposition~\ref{prop:bound-Kn}).

%The approach was similar to the present work, with some
%specific argument that permitted the improvement.

Let us briefly indicate the strategy.
It is classical that  $K_n(\O_F)$ relates directly
to the homology of $\GL_N(\O_F)$, for $N = 2n+1$ (see
Section~\ref{sec:comparison-torsion-homology}). 
The general idea for the proof of Theorem~\ref{thm:main} is to obtain an
upper bound for the homology of $\GL_N(\O_F)$ by using its action on the
symmetric space $X = \PGL_N(\C)/\PU(N)$. A theorem of Gelander 
(Theorem~\ref{thm:Gelander}) permits to control the topology of
noncompact arithmetic quotients of $X$ in terms of their volume, and from
this we can obtain an upper bound for the torsion homology of
$\PGL_N(\O_F)$ (Section~\ref{sec:homology-PGL}). 
A spectral sequence argument finally provides the bound
for $\tors\, H_n(\GL_N(\O_F))$.
Unfortunately the constant $C(n)$ in \eqref{eq:main-thm} is not explicit, 
and its appearance is explained by some nonexplicit constant in Gelander's theorem.

Our result can be compared with the upper bounds obtained by 
Soul\'e in \cite[Prop.4~iii)]{Soule03}.
He showed that for $F$ imaginary quadratic one has ``up to small torsion'':
\begin{align}
        \log|\tors\, K_n(\O_F)| &\le |D_F|^{1120n^4 \log(n)}. 
        \label{eq:Soule-imag-quad}
\end{align}
Thus, Theorem~\ref{thm:main} improves (asymptotically) the upper bound for 
the existence of $p$-torsion in $K_n(\O_F)$ that follows from 
\eqref{eq:Soule-imag-quad}.
Soul\'e's work also contains general upper bounds -- without restriction on the signature of the
number field $F$ -- which have been later improved in \cite{BEmH}.
It is not clear if the strategy for Theorem~\ref{thm:main} could be adapted to include the case 
of number fields with more than one archimedean place (see Remark~\ref{rmk:other-fields}).
Until Section \ref{sec:GL-PGL} and the last step in the proof, we will be treating
the case of general number fields $F$.

A main motivation for the study of upper bounds in $K$-theory comes from the case
$F = \Q$, for which the triviality of $K_{4m}(\Z)$ is known to be
equivalent to the Vandiver conjecture (cf. \cite{Soule99}). To our knowledge, the best available
bounds  (up to small torsion) for these groups are given
in \cite[Prop.~4 iv)]{Soule03}. The method for 
Theorem~\ref{thm:main} applies directly to the case $F = \Q$, and
an effective version of Gelander's theorem would thus provide an upper bound for $K_n(\Z)$.
However, it might well be that the bounds obtained this way will be
larger than Soul\'e's bounds.

\section{Preliminaries}
\label{sec:comparison-torsion-homology}

\subsection{}
\label{sec:tors-ell}

For an abelian group $A$, let $\tors\, A$ denote its torsion subgroup.
For $\ell > 1$, we define

\begin{align}
        \tors_\ell\, A &= \tors \left( A \otimes \Z\big[1/m\big] \right),
        \label{eq:tors-ell}
\end{align}
where $m$ is the lowest multiple common to all integers $\le \ell$.
Thus, $\tors_\ell$ kills all torsion up to $\ell$.

%\begin{lemma}
  %\label{lemma:kills-p}
  %Let $p > \ell$ be prime. If $\tors_\ell(A)$ has no $p$-torsion, then
  %the same holds in $A$. 
%\end{lemma}

\subsection{}
\label{sec:prelim-tors-homol}

Let $\Gamma$ be a group.
The symbol $H_n(\Gamma)$ will denote the homology of $\Gamma$ with coefficients in $\Z$. 
Since $\Z\left[ 1/m \right]$ is a flat $\Z$-module, it follows from the
universal coefficient theorem \cite[p.~8]{Brown82} that 
\begin{align}
  \tors_\ell\, H_n(\Gamma) & \cong H_n(\Gamma, \Z[1/m]), 
\end{align}
for $m$ as above. In this context, we will also need the following fact.

\begin{lemma}
     \label{lemma:homology-finite-index-sbgp}
   Let $\Gamma_0 \subset \Gamma$ be a normal subgroup of index $[\Gamma :
     \Gamma_0] \le \ell < \infty$, and $m$ be the lowest multiple common to all
     integers $\le \ell$.
     Then $H_n(\Gamma, \Z[1/m])$ injects into $H_n(\Gamma_0, \Z[1/m])$.
\end{lemma}

\begin{proof}
  Let $A = \Gamma / \Gamma_0$. 
  Since $|A| = [\Gamma:\Gamma_0]$ is invertible in $\Z[1/m]$,
  the transfer map (cf. \cite[\S III.9--10]{Brown82}) gives an
  isomorphism between $H_n(\Gamma, \Z[1/m])$ and the module of
  co-invariants $H_n(\Gamma_0, \Z[1/m])_A$. But for $|A|$ invertible,
  the norm map (cf. \cite[\S III.1]{Brown82}) provides an isomorphism of 
  the latter with the submodule  $H_n(\Gamma_0, \Z[1/m])^A$ of
  $A$-invariant elements in  $H_n(\Gamma_0, \Z[1/m])$.
\end{proof}

\subsection{}

$F$ denotes a number field of degree $d$. Until Section \ref{sec:GL-PGL} there is 
no restriction on $d$. 

The group $K_n(\O_F)$ is defined as the $n$-th
homotopy group $\pi_n(B\GL(\O_F)^+)$ of Quillen's plus contruction
applied to the classifying space of the linear group $\GL(\O_F) =
\varinjlim \GL_N(\O_F)$. Since the (integral) homology of $B \GL(\O_F)^+$
is canonically isomorphic to the homology of $\GL(\O_F)$, the Hurewicz 
map provides a homomorphism 
\begin{align}
  K_n(\O_F) &\to H_n(\GL(\O_F)).
  \label{eq:Hurew-map}
\end{align}
%This comparison between $K$-theory and homology is the starting point
%for the computation of the rank of $K_n(\O_F)$ by Borel in
%\cite{Bor74}.
%It also permits some precise computation of $K_n(\Z)$ for $n$ small (cf.
%REF).

The kernel of the Hurewicz map \eqref{eq:Hurew-map}
does not contain $p$-torsion elements, for any $p
> \frac{n+1}{2}$ (see \cite{Arl91}).
It follows that for $\ell \ge \frac{n+1}{2}$,
the group $\tors_{\ell}\, K_n(\O_F)$
injects into $\tors_\ell\, H_n(\GL(\O_F))$.
Moreover, the stability result of Maazen (see van der Kallen
\cite[Theorem 4.11]{vdK80}) tells us that the homology group
$H_n(\GL(\O_F))$ equals $H_n(\GL_N(\O_F))$ for any $N \ge 2n+1$. 
These various results provide the next proposition,
which is already a basic ingredient in Soul\'e's method \cite{Soule03}. 

\begin{prop}
  \label{prop:Kn-H-GL}
  For $N \ge 2n+1$ and $\ell \ge \frac{n+1}{2}$, there is an injective map
  \begin{align*}
  \tors_{\ell}\, K_n(\O_F) \to \tors_\ell\, H_n(\GL_N(\O_F)).
  \end{align*}
\end{prop}

\section{Torsion homology of $\PGL_N(\O_F)$}
\label{sec:homology-PGL}

\subsection{}
\label{sec:gamma}

Consider the semisimple Lie group $G = \PGL_N(F \otimes_\Q \R)$. The arithmetic subgroup
$\Gamma = \PGL_N(\O_F)$ is a nonuniform (irreducible) lattice in $G$. 
The following fact is well-known (see for instance \cite[Lemma 13.1]{Gel04}): 
\begin{lemma}
  \label{lemma:const-gamma}
  There exist a constant $\gamma = \gamma(d,N)$ such that for any number field $F$ of degree $d$ the
  group $\Gamma = \PGL_N(\O_F)$ has a torsion-free normal subgroup $\Gamma_0$
  with index $[\Gamma : \Gamma_0 ] \le \gamma$.
\end{lemma}
The result essentially follows from Minkowski's lemma that  asserts
that the reduction map $\GL_n(\Z) \to \GL(\Z / m)$ has torsion-free 
kernel for $m > 2$. This permits to give an explicit value for the 
bound $\gamma$.  See \cite[Sect. 2]{BelEm14} for an modified argument
that provides a better (i.e., lower) value for this constant.

%It is well-known (see for instance \cite[Lemma 13.1]{Gel04}) that we can find a torsion-free subgroup $\Gamma_0
%\subset \Gamma$ whose index $[\Gamma:\Gamma_0]$ is uniformly bounded by a constant $\gamma(d)$,
%which depends only on the degree of $F$. This essentially follows from Minkowski's lemma that asserts
%that the reduction map $\GL_n(\Z) \to \GL(\Z / m)$ has torsion-free kernel for $m > 2$.

\subsection{}

Let $X$ be the symmetric space associated with $G$, i.e, $X = G/K$ for some maximal compact subgroup 
$K \subset G$. Note that $X$ depends only on $N$ and the signature of $F$.
For $\Gamma_0$ as in Lemma~\ref{lemma:const-gamma}
the quotient $M = \Gamma_0 \bs X$ is a noncompact locally symmetric space. Gelander's  
theorem \cite[Theorem 1.5 (1)]{Gel04} shows that $M$ is homotop to a simplicial complex
of bounded size. Let us define a {\em $(\delta,
v)$-simplicial complex} to be a simplicial complex with at most $v$
vertices, all of valence at most $\delta$.  Then the precise result is the
following.

\begin{theorem}[Gelander]
  %Let $X$ be a symmetric space of noncompact type.
  There exist  constants $\delta = \delta(X)$ and $\alpha = \alpha(X)$ such that 
  any noncompact arithmetic manifold $M = \Gamma_0\bs X$ is homotopically equivalent to a
  $(\delta, \alpha \vol(M))$-simplicial complex $\mathcal K$.
  \label{thm:Gelander}
\end{theorem}

\begin{rmk}
  \label{rmk:normalization-volume}
  It is clear that the result holds for any normalization of the volume
  $\vol$ on $X$, by adapting the constant $\alpha$ accordingly.
\end{rmk}

A result of Gabber allows to bound the torsion homology of a simplicial
complex. We will use this result in the following form. 

\begin{prop}
  \label{prop:Gabber}
  Let $\K$ be a $(\delta, v)$-simplicial complex. Then the torsion
  homology of $\K$ is bounded  as follows: 
  \begin{align}
    \nonumber \log |\tors\, H_n(\K)| &\le v \cdot \binom{\delta}{n} \\
    &\le v \cdot \delta^n
    \label{eq:bound-homology-K}
  \end{align}
\end{prop}

\begin{proof}
  The number of $n$-simplices in $\K$ is bounded above by $\frac{v}{n+1}
  \binom{\delta}{n}$. By Gabber's lemma (cf.\ \cite[Lemma 2.2]{EmeK2}),
  since the chain complex $\K_\bullet$ is a \emph{simplicial} complex we may bound the torsion 
  homology as follows:
  \begin{align*}
    |\tors \, H_n(\K_\bullet)| &\le \sqrt{n+2}^{\, \mathrm{rank}(\K_n)},
  \end{align*}
  from which the result follows.
\end{proof}

\subsection{}
From Theorem~\ref{thm:Gelander} and Proposition~\ref{prop:Gabber} we obtain
\begin{align}
  \log |\tors\, H_n(\Gamma_0)| &\le \alpha  \delta^n \; \vol(\Gamma_0 \bs
  X)
  \nonumber \\
  &\le \alpha \delta^n \gamma \; \vol(\Gamma \bs X),
  \label{eq:bound-H-Gamma_0}
\end{align}
for constants $\alpha, \gamma$ and $\delta$ that can be chosen to
depend only on $N$ and $d$.

\begin{prop}
  \label{prop:homology-PGL}
  Let $N = 2n+1$
  and $\Gamma = \PGL_N(\O_F)$. For some constant $C(d, n)$, we have the following upper bound for
  the torsion homology:
  \begin{align}
    \log |\tors_\gamma\, H_n(\Gamma)| &\le C(d,n) |D_F|^{2 n(n+1)},
    \label{eq:bound-homology-PGL}
  \end{align}
  where $\gamma = \gamma(d, 2n+1)$ is the constant in
  Lemma~\ref{lemma:const-gamma}.
\end{prop}

\begin{proof}
  By Section~\ref{sec:prelim-tors-homol} and
  Lemma~\ref{lemma:homology-finite-index-sbgp}, it is enough to bound 
  \begin{align*}
    |\tors\,H_n(\Gamma_0, \Z[1/m])| &= |\tors\, H_n(\Gamma_0) \otimes
    \Z[1/m]|\\
    &\le |\tors\, H_n(\Gamma_0)|.
  \end{align*}
  Using~\eqref{eq:bound-H-Gamma_0}, it remains to bound the covolume of
  $\Gamma = \PGL_N(\O_F)$. It is certainly bounded by the covolume of $\PSL_N(\O_F)$. 
  We can choose (cf. Remark~\ref{rmk:normalization-volume}) 
  to work with the  normalization of the volume used in Prasad's volume formula: 
  let us assume that 
  \begin{align}
          \vol(\PSL_N(\O_F)\bs X) = \mu(\SL_N(\O_F) \bs \tilde{G}),
          \label{eq:normal-vol}
  \end{align}
  where $\mu = \mu_\infty = \mu_S$ is the Haar measure on $\tilde{G} = \SL_N(F \otimes_\Q \R)$
  defined in \cite[\S 3.6]{Pra89}. Then the volume formula \cite[Theorem 3.7]{Pra89} shows that 
  \begin{align*}
    \vol(\PSL_N(\O_F) \bs X) &=  A \cdot |D_F|^{\frac{N^2-1}{2}} \prod_{k=2}^{N} \zeta_F(k),
    \end{align*}
  for some (explicit) constant $A$ that depends only on $N$ and the signature of $F$. 
  The product of zeta functions is easily shown to be $\le 2^{N-1}$, so that
  the result follows.
\end{proof}

\begin{rmk}
  Note that if we had explicit constants $\alpha$ and $\delta$ in
  Theorem~\ref{thm:Gelander} (for some
  explicit normalization of the volume), the procedure  would permit to
  make the constant $C(d, n)$ explicit.
\end{rmk}

\section{Upper bounds for the $K$-groups}
\label{sec:GL-PGL}

We finally state and prove the following result, of which
Theorem~\ref{thm:main} will be an immediate consequence (see proof below).
Recall that the notation $\tors_\gamma$ was introduced in Section~\ref{sec:tors-ell}.

\begin{prop}
  Let $n \ge 2$.
  There are constants $C(n)$ and $\gamma = \gamma(n)$  such that for any imaginary 
  quadratic field $F$, one has
  \begin{align}
    \log|\tors_\gamma\, K_n(\O_F)| &\le C(n) |D_F|^{2n(n+1)}.
    \label{eq:bound-Kn}
  \end{align}
  \label{prop:bound-Kn}
\end{prop}

\begin{proof}
  With $\gamma$ as in Lemma~\ref{lemma:const-gamma} and larger than
  $\frac{n+1}{2}$, we have from Proposition~\ref{prop:Kn-H-GL}:
  \begin{align}
    |\tors_\gamma\, K_n(\O_F)| &\le |\tors_\gamma\, H_n(\GL_N(\O_F))|
    \nonumber \\
    &= |H_n(\GL_N(\O_F), \Z[1/m])|,
  \end{align}
  with $N = 2n+1$ and $m$ the lowest multiple common to all integers
  $\le \gamma$.

  With $\Gamma = \PGL_N(\O_F)$, we have the short exact sequence 
  \begin{align}
   1 &\to \O_F^\times \to \GL_N(\O_F) \to \Gamma \to 1, 
    \label{eq:exact-seq-PGL}
  \end{align}
  and for it the Lyndon-Hochschild-Serre spectral sequence (cf.\
  \cite[\S VII.6]{Brown82}) reads: 
  \begin{align}
    E_{pq}^2 = H_p\left( \Gamma, H_q(\O_F^\times, \Z[1/m]) \right)  &\implies
    H_{p+q}\left(\GL_N(\O_F), \Z[1/m] \right).
    \label{eq:spect-seq-2}
  \end{align}
  Since we assume that $F$ is imaginary quadratic, we have that
  $\O_F^\times$ is finite, of order $\le 6$. Then $H_q(\O_F^\times,
  \Z[1/m]) = 0$ for $q > 0$, and the sprectral sequence is concentrated
  in $q = 0$. Note also that since the image of $\O_F^\times$ is central
  in $\GL_N(\O_F)$, the action of $\Gamma$ on $H_0(\O_F^\times, \Z[1/m]) =
  \Z[1/m]$ is trivial (cf. \cite[ex.~1 p.80]{Brown82}). We obtain:
  \begin{align*}
    H_n(\Gamma, \Z[1/m]) &= H_n(\GL_N(\O_F), \Z[1/m]),
  \end{align*}
  and the result follows from Proposition~\ref{prop:homology-PGL} with
  the same constant $C(n) = C(2, n)$.
\end{proof}

\begin{rmk}
        As it follows from the discussion in Section~\ref{sec:gamma}, the constant $\gamma$ can 
        be easily made explicit -- contrarily to $C(n)$. 
\end{rmk}

\begin{rmk}
        \label{rmk:other-fields}
        For $F$ with more than one archimedean place, the spectral sequence~\eqref{eq:spect-seq-2} 
        does not collapse at $E^2_{pq}$. A priori the differential $\del_r$ (with $r \ge 2$) might 
        add torsion, whenever both its domain and image are infinite. It seems difficult to control 
        this phenomena to obtain upper bounds in this more general setting.
\end{rmk}

\begin{proof}[Proof of Theorem~\ref{thm:main}]
  In Proposition~\ref{prop:bound-Kn}, we can assume that $C(n) \ge
  \log(\gamma)$ (by increasing the constant if necessary).
  Let $s \in K_n(\O_F)$ be an element of oder $p$, with $p$ prime such
  that $\log(p) > C(n) |D_F|^{2n(n+1)}$. In particular, 
  we have  $p > \gamma$, so that the image of $s$  in
  $\tors_\gamma\, K_n(\O_F)$ is non-trivial. But this contradicts the
  bound~\eqref{eq:bound-Kn}.
\end{proof}

\bibliographystyle{amsplain}
\bibliography{Kn-PGL.bbl}

\end{document}